\documentclass[8pt,a4paper]{article}

\usepackage{amsmath,amsfonts,amssymb}
\usepackage{bbm}
\usepackage{cases}
\usepackage{kotex}
\usepackage{lipsum}
\usepackage{subcaption}
\usepackage{algorithm}
\usepackage{algpseudocode}
\usepackage{xcolor}
\usepackage{graphicx}
\usepackage{ulem}
\usepackage{subcaption}
\newtheorem{defn}{Definition}[section]
\newtheorem{theo}[defn]{Theorem}
\newtheorem{lem}[defn]{Lemma}
\newtheorem{prop}[defn]{Proposition}

\newtheorem{rem}[defn]{Remark}

\newenvironment{proof}{{\bf Proof }}{{\vskip 0.1cm \hfill$\Box$}}
\begin{document} 

\noindent
{\Large \bf Cholesky decomposition and well-posedness of Cauchy problem for Fokker-Planck equations with unbounded coefficients}
\\ \\
\bigskip
\noindent
{\bf Haesung Lee}  \\
\noindent
{\bf Abstract.} This paper explores the well-posedness of the Cauchy problem for the Fokker-Planck equation associated with the partial differential operator $L$ with low regularity condition. To address uniqueness, we apply a recently developed superposition principle for unbounded coefficients, which reduces the uniqueness problem for the Fokker–Planck equation to the uniqueness of solutions to the martingale problem. Using the Cholesky decomposition algorithm, a standard tool in numerical linear algebra, we construct a lower triangular matrix of functions $\sigma$ with suitable regularity such that $A = \sigma \sigma^T$. This formulation allows us to connect the uniqueness of solutions to the martingale problem with the uniqueness of weak solutions to It\^{o}-SDEs. For existence, we rely on established results concerning sub-Markovian semigroups, which enable us to confirm the existence of solutions to the Fokker–Planck equation under general growth conditions expressed as inequalities. Additionally, by imposing further growth conditions on the coefficients, also expressed as inequalities, we establish the ergodicity of the solutions. This work demonstrates the interplay between stochastic analysis and numerical linear algebra in addressing problems related to partial differential equations. 


\section{Introduction} %\label{intro}
This paper explores the existence and uniqueness of solutions to the Fokker-Planck equation, which is an important topic across various fields, including the theory of partial differential equations, stochastic analysis (cf. \cite{BKRS15, P14}), and recently, in generative models that have been actively studied in artificial intelligence (cf. \cite{A82, SS21}). To describe the Fokker-Planck equation, we begin with the following  local regularity condition:  \\ \\
{\bf (H-1):} {\it $d \geq 2$, $\mathbf{G} \in L^p_{\text{loc}}(\mathbb{R}^d, \mathbb{R}^d)$ with $p \in (d, \infty)$, and $A = (a_{ij})_{1 \leq i,j \leq d}$ is a symmetric matrix of functions such that $a_{ij} \in H^{1,p}_{\text{loc}}(\mathbb{R}^d) \cap C(\mathbb{R}^d)$ for all $1 \leq i,j \leq d$. Moreover, $A$ is locally uniformly strictly elliptic on $\mathbb{R}^d$, i.e. for each open ball $B$, there exist strictly positive constants $\lambda_B$ and $\Lambda_B$ such that  
    \[
    \lambda_B \| \xi \|^2 \leq \langle A(x) \xi, \xi \rangle \leq \Lambda_B \|\xi \|^2 \quad \text{for all } x \in B \text{ and } \xi \in \mathbb{R}^d.
    \]  
$L$ is a partial differential operator on $C^2(\mathbb{R}^d)$ defined as  
\[
Lf = \frac{1}{2} \text{\rm trace}(A \nabla^2 f) + \langle \mathbf{G}, \nabla f \rangle, \quad f \in C^2(\mathbb{R}^d).
\]  
}
\centerline{}
In Definition \ref{defnofoker}, we present a definition of solution to the Cauchy problem for the Fokker–Planck equation associated with $L$ and the initial distribution $\delta_x$, which is also well-explained in the Introduction of \cite{BRS21} and \cite[Proposition 6.1.2]{BKRS15}.
If the coefficients of $L$ are smooth and satisfy mild growth conditions expressed as inequalities, then the existence and uniqueness of solutions to the Cauchy problem for the Fokker–Planck equation associated with $L$ are well-known, as established in \cite[Theorem 4.1]{P14}.
Indeed, the results mentioned in \cite[Theorem 4.1]{P14} ensure the twice continuous differentiability of the solution's density, which allow the solution to be considered as a classical solution. However, when the coefficients are non-smooth, such classical solutions may not be expected. Consequently, solutions defined in the distributional sense, as introduced in Definition \ref{defnofoker}, are required. This area of research has been extensively studied by Bogachev, Krylov, R\"{o}ckner, and Shaposhnikov, with comprehensive coverage in their monograph \cite{BKRS15}. Specifically, it is shown in \cite[Theorems 6.6.4, 9.4.6]{BKRS15} that if the components of $A$ are locally Lipschitz continuous, $\mathbf{G} \in L^{p+2}_{\text{loc}}(\mathbb{R}^d, \mathbb{R}^d)$ with $p \in (d, \infty)$, and certain growth conditions on the coefficients, expressed through Lyapunov functions, are satisfied, then the existence and uniqueness of the Fokker-Planck equation associated with $L$ and initial distribution $\delta_x$ hold.
\\
One of the most important aspects of their argument in \cite{BKRS15} is the use of the fact that the solution to the Fokker-Planck equation is absolutely continuous with respect to the Lebesgue measure, enabling us to use arguments to show that the ratio of the two densities of solutions equals one. Actually, independent of existence and uniqueness, the study of the regularity of solutions under the assumption of a priori existence of a solution has been well-established by \cite{BKR01}. Moreover, \cite[Theorem 4.1]{BKR01} shows the existence of a solution to the Cauchy problem for the Fokker–Planck equation associated with $L$ and the initial distribution $\delta_x$, under the assumption that $a_{ij} \in H^{1,p+2}_{loc}(\mathbb{R}^d)$ for all $1 \leq i,j \leq d$ and $\mathbf{G} \in L^{p+2}_{loc}(\mathbb{R}^d, \mathbb{R}^d)$, in addition to {\bf (H-1)}. Research utilizing densities with appropriate regularity to demonstrate uniqueness has also been extensively explored in \cite{BPRS07, BKS21, Sh12}. \\
Recently, several attempts using stochastic analysis have been made to approach the existence and uniqueness of solutions to Fokker-Planck equations, without relying on regularity information about the solutions. One of the most important tools for this is the superposition principle. To explain it, we first introduce the concept of the martingale problem up to a finite time $T$ in Definition \ref{martiprupto}, which is widely studied in stochastic analysis, as described in the introduction of \cite{BRS21}. The superposition principle states that, under appropriate conditions on the coefficients of $L$, if $(\nu_t)_{t \in [0,T]}$ is a solution to the Cauchy problem for the Fokker-Planck equation associated with $L$ and the initial distribution $\delta_x$, then there exists a corresponding solution $\bar{\mathbb{P}}_x$ to the martingale problem for $(L, C_0^{\infty}(\mathbb{R}^d))$ up to a finite time $T$ as in Definition \ref{martiprupto} such that 
$$
\bar{\mathbb{P}}_x ( \bar{X}_t \in E ) = \nu_t(E) \quad \text{ for any $E \in \mathcal{B}(\mathbb{R}^d)$ and $t\in [0,T]$}.
$$
This superposition principle is introduced in \cite[Theorem 2.6]{Fi08} where the coefficients of $L$ are required to be global bounded, and later \cite{Tr16} slightly relax the global bounded conditions on the coefficietns. For instance,  by using \cite[Theorem 2.6]{Fi08}  and \cite[Theorem 6.3.4]{SV06}, one can derive the following uniqueness result: if $A=\sigma \sigma^T$,  all components of $\sigma=(\sigma_{ij})_{1 \leq i,j \leq d}$ and $\mathbf{G}=(g_1, \ldots, g_d)$ are globally bounded and satisfies that for some constant $C>0$
$$
\sum_{i,j=1}^d|\sigma_{ij}(x)-\sigma_{ij}(y)|+\| \mathbf{G}(x) - \mathbf{G}(y) \| \leq C \| x- y\| \quad \text{ for all $x,y \in \mathbb{R}^d$},
$$
where $\|\cdot \|$ denotes the Euclidean norm in $\mathbb{R}^d$, then the uniqueness of solutions to the Cauchy problem for the Fokker-Planck equation associated with $L$ and the initial distribution $\delta_x$ holds (see \cite[Theorem 9.8.5]{BKRS15}). Similar uniqueness results were handled in \cite{RZ10} by showing the pathwise uniqueness. However, this type of uniqueness results derived through superposition principle and stochastic analysis do not explicitly guarantee the existence of solutions to the Fokker-Planck equations. Therefore, to establish the existence of solutions, a separate and independent discussion apart from uniqueness is required.\\
As a recent result established in \cite{LT21} (cf. \cite{LST22}), an analytic approach to stochastic differential equations has been successfully developed. As an intermediate result of this approach, under the assumption {\bf (H-1)}
it was proven in  \cite[Proposition 3.13(iii)]{LST22} (cf. \cite[Lemma 3.14(iii)]{LT21}) that there exists a sub-Markovian semigroup $(P_t)_{t>0}$ satisfying  that for each $\varphi \in C_0^{\infty}(\mathbb{R}^d)$
\[
P_t \varphi(x) - \varphi(x) = \int_0^t P_s L \varphi(x) \, ds \quad \text{for all } t \in (0, \infty) \text{ and } x \in \mathbb{R}^d.
\]  
Furthermore, by establishing additional continuity of $(P_t)_{t>0}$ at $t=0$ as in Proposition \ref{continuprop}, the existence of a solution to the Cauchy problem for the Fokker-Planck equation with the initial distribution $\delta_x$ was demonstrated, where the details was explained in Theorem \ref{mainexisthm}. 
To establish the uniqueness, we turn our focus back to stochastic analysis based on the superposition principle. The superposition principle developed in \cite{Fi08, Tr16} imposes strict growth conditions on the coefficients. However, the more recent superposition principle \cite{BRS21} developed by Bogachev, R\"{o}ckner, and Shaposhnikov not only allows coefficients with mild growth but also permits locally unbounded drift coefficients. Now, to achieve uniqueness for solutions to the Cauchy problem for the Fokker–Planck equation associated with $L$ and the initial distribution $\delta_x$, we need several steps. First, we construct a solution $\bar{\mathbb{P}}_x$ which solves the martingale problem up to a finite time $T$ by using the superposition principle in \cite[Theorem 1.1]{BRS21}. Then, in Proposition \ref{prop12}, we extend the solution $\bar{\mathbb{P}}_x$ to a probability measure on $C([0, \infty), \mathbb{R}^d)$.
\noindent
We then show in Theorem \ref{cholesky} the existence of a matrix of functions $\sigma=(\sigma_{ij})_{1 \leq i,j \leq d}$ with $\sigma_{ij} \in H^{1,p}_{\text{loc}}(\mathbb{R}^d) \cap C(\mathbb{R}^d)$ for all $1 \leq i,j \leq d$ such that $A = \sigma \sigma^T$, where $A$ is the matrix of functions given in {\bf (H-1)}. One of the conventional approaches is to find a symmetric matrix of functions $\sigma=(\sigma_{ij})_{1 \leq i,j \leq d}$ with $A=\sigma^2$ as in \cite[Chapter 6, Lemma 1.1]{F06} (cf. \cite[Lemma 2.1]{CH97}). However, the symmetric matrix of functions $\sigma$ as described in \cite[Chapter 6, Lemma 1.1]{F06} is defined as complex-valued contour integral for a matrix of functions, and hence it is not easy to check whether $\sigma_{ij} \in H^{1,p}_{loc}(\mathbb{R}^d) \cap C(\mathbb{R}^d)$ for all $1 \leq i,j \leq d$. To resolve this, we utilize the Cholesky decomposition (\cite[Section 6.3]{AK08}), which is widely employed in numerical linear algebra. Specifically, using the Cholesky decomposition, the process of finding a unique lower triangular matrix $\sigma=(\sigma_{ij})_{1 \leq i,j \leq d}$ satisfying $A = \sigma \sigma^T$ and $\sigma_{ii}>0$ for all $i=1, \ldots, d$, proceeds algorithmically. This process is carried out sequentially in the order of the column indices of $\sigma$, and within each column, calculations are performed in the order of row indices. At each step, each $\sigma_{ij}$ is expressed in an algebraic form based on the previously computed data. Consequently, it can be shown algorithmically and inductively that $\sigma_{ij} \in H^{1,p}_{\text{loc}}(\mathbb{R}^d) \cap C(\mathbb{R}^d)$ for all $1 \leq i,j \leq d$ (see Theorem \ref{cholesky} for details). As the final step for uniqueness, we apply Ikeda-Watanabe's theorem (\cite[Chapter II, Theorem 7.1$^\prime$]{IW89}) to identify the solution to the martingale problem defined in Proposition \ref{prop12} as a weak solution to the corresponding It\^{o}-stochastic differential equation (It\^{o}-SDE). Then, by utilizing the well-established pathwise uniqueness results for It\^{o}-SDEs as in Theorem \ref{ouruniquene}, we can invoke Yamada-Watanabe's theorem (cf. \cite[Chapter 5, Proposition 3.20]{KS91}) (for original results, see \cite[Proposition 1]{YW71}) to establish uniqueness in law. This ultimately leads to the uniqueness of solutions to the Cauchy problem for the Fokker-Planck equations associated with $L$ and the initial distribution $\delta_x$. \\
Now, before presenting our main results in this paper, we additionally consider the following condition:\\ \\
{\bf (H-2):} {\it There exists a constant $K>0$ and $N_0 \in \mathbb{N}$ such that
$$
\| A(x)\| \leq K + K\|x\|^2 \ln (1+\|x\|^2) \quad \text{ for all $x \in \mathbb{R}^d$}
$$
and
\begin{align*}
\langle \mathbf{G}(x), x \rangle \leq K + K \|x\|^2 \ln (1+\|x\|^2) \quad \text{ for a.e. $x \in \mathbb{R}^d \setminus \overline{B}_{N_0}$}.
\end{align*}
(Here, $\|A \|$ denotes the operator norm of $A$, i.e. $\|A\|:=\sup \{\|A \xi\|: \|\xi \| \leq1  \}$).
}
\centerline{}
\centerline{}
\begin{theo} \label{mainresu}
Let $x \in \mathbb{R}^d$, and assume that {\bf (H-1)} and {\bf (H-2)} hold. Then, the following hold:
\begin{itemize}
\item[(i)]
There exists a family of probability measures $(\mu_t)_{ t \in [0, \infty)}$ on $\mathcal{B}(\mathbb{R}^d)$ such that $(\mu_t)_{t \in [0, \infty)}$ is a solution to the Cauchy problem for the Fokker–Planck equation associated with $L$ and the initial distribution $\delta_x$ as in Definition \ref{defnofoker}.
\item[(ii)]
$(\mu_t)_{ t \in [0, \infty)}$ in (i) is a unique solution to the Cauchy problem for the Fokker–Planck equation associated with $L$ and the initial distribution $\delta_x$ in the following sense: Let $T>0$ be arbitrarily given. Then, if a family of probability measures $(\nu_t)_{ t \in [0, T]}$ on $\mathcal{B}(\mathbb{R}^d)$ is a solution to the Cauchy problem for the Fokker–Planck equation associated with $L$ and the initial distribution $\delta_x$ as in Definition \ref{defnofoker}, then
$$
\mu_t = \nu_t, \quad \text{ on $\mathcal{B}(\mathbb{R}^d)$} \quad \text{ for each $t \in [0,T]$}.
$$
\item[(iii)]
Let $(\mu_t)_{ t \in [0, \infty)}$ be as in (i).  If there exist constants $M>0$ and $N_0 \in \mathbb{N}$ such that either the inequality \eqref{sprininva1} or \eqref{sprininva2} holds, then there exists a probability measure $\tilde{\mu}$ on $\mathcal{B}(\mathbb{R}^d)$ such that
$$
\int_{\mathbb{R}^d} Lf d\tilde{\mu}=0 \quad \text{ for all $f \in C_0^{\infty}(\mathbb{R}^d)$}
$$
and
\begin{align*}
\lim_{t \rightarrow \infty} \mu_t(E) = \tilde{\mu}(E) \quad \text{ for any $E \in \mathcal{B}(\mathbb{R}^d)$}.
\end{align*}
\end{itemize}
\end{theo}
\noindent
The proof of Theorem \ref{mainresu}(i) is described as in the one of Theorem \ref{mainexisthm}(iv) and Remark \ref{growthcondi}(i), where $\mu_t(dy):=P_t(x,dy)$. The proof of Theorem \ref{mainresu}(ii) is addressed in the one of Theorem \ref{ouruniquene}. Finally, the proof of Theorem \ref{mainresu}(iii) is presented in the one of Theorem \ref{mainexisthm}(v) and Remark \ref{growthcondi}(ii), where $\tilde{\mu}:=\frac{1}{\mu(\mathbb{R}^d)} \mu$. 
The novelty of this paper lies in combining the results of \cite{LT21, LST22, BRS21} to establish the well-posedness of Fokker–Planck equations associated with operators whose drift terms are locally unbounded and not weakly differentiable, which are cases that could not be covered by \cite{Fi08} and \cite{Tr16}.  In addition, we establish the ergodicity of solutions, which further strengthens the applicability of our results to future developments in applied areas such as MCMC algorithms \cite{HHS05, MCF15, SS21}, especially in contexts involving more general operators and target measures described by limiting distributions.
\\
The structure of this paper is as follows. In Section \ref{basicnota}, we introduce the fundamental notations that will be used throughout the paper. Section \ref{exisergo} focuses on establishing the existence of solutions to the Cauchy problem for the Fokker-Planck equations associated with $L$ and the initial distribution $\delta_x$. In Section \ref{unibystoc}, we address the uniqueness of solutions by employing Cholesky decomposition and stochastic analysis. Finally, Section \ref{condis} concludes the paper with a brief discussion.

\section{Basic notations and definitions} \label{basicnota}
\noindent
In this study, we work within the Euclidean space $\mathbb{R}^d$ for $d \geq 2$, equipped with the standard Euclidean inner product $\langle \cdot, \cdot \rangle$ and the corresponding Euclidean norm $\|\cdot\|$. An open ball centered at $x_0 \in \mathbb{R}^d$ with radius $r > 0$ is defined as $B_R(x_0) := \{x \in \mathbb{R}^d : \|x - x_0\| < R\}$. For real numbers $a, b \in \mathbb{R}$, the notations $a \wedge b := \min(a, b)$ and $a \vee b := \max(a, b)$ are used. Let $U$ be an open subset of $\mathbb{R}^d$. The notation $\mathcal{B}(U)$ denotes the set of all Borel measurable sets or functions on $U$, as appropriate. The Lebesgue measure on $\mathcal{B}(\mathbb{R}^d)$ is denoted by $dx$. For fixed $x \in \mathbb{R}^d$, $\delta_x$ denotes the dirac delta measure on $\mathcal{B}(\mathbb{R}^d)$ centered at $x$. For a subset $\mathcal{A} \subset \mathcal{B}(U)$, the subset $\mathcal{A}_0$ consists of functions $f \in \mathcal{A}$ such that $\text{supp}(f \cdot dx)$ is compact and contained in $U$. The set of continuous functions on $U$ and its closure $\overline{U}$ are denoted by $C(U)$ and $C(\overline{U})$, respectively. Additionally, $C(U)_0$ are denoted as $C_0(U)$. For $k \in \mathbb{N} \cup \{\infty\}$, the space of $k$-times continuously differentiable functions on $U$ is denoted by $C^k(U)$, and $C^k_0(U) := C^k(U) \cap C_0(U)$. \\
%For $\beta \in (0,1]$, the Hölder space $C^{0, \beta}(\overline{U})$ consists of continuous functions $f$ satisfying 
%$$
%\sup_{x,y \in \overline{U}} \frac{|f(x) - f(y)|}{\|x-y\|^{\beta}} < \infty. 
%$$
%The localized version $C^{0, \beta}_{\text{loc}}(U)$ denotes the set of all functions $f \in C^{0, \beta}(\overline{V})$ for any bounded open set $V$ with $\overline{V} \subset U$. 
%For functions defined on $\overline{U}$, $C^k(\overline{U})$ contains those functions $f$ for which there exist an open set $V \supset \overline{U}$ and a function $\tilde{f} \in C^k(V)$ such that $\tilde{f} = f$ on $\overline{U}$.\\ \\
Let $r \in [1, \infty]$. The $L^r$-space on $U$ with respect to a measure $\nu$ is denoted by $L^r(U, \nu)$, equipped with the standard $L^r(U, \nu)$-norm. Similarly, $L^r(U, \mathbb{R}^d, \nu)$ denotes the space of $L^r$-vector fields on $U$ with norm $\|\mathbf{F}\|_{L^r(U, \nu)} := \big\|\|\mathbf{F}\|\big\|_{L^r(U, \nu)}$. For localized $L^r$-spaces, $L^r_{\text{loc}}(U, \nu)$ denotes the set of all Borel measurable functions $f$ such that $f|_{W} \in L^r(W, \nu)$ for any bounded open set $W \subset \mathbb{R}^d$ with $\overline{W} \subset U$. The set of all vector fields $\bold{F}$ on $U$ satisfying $\|\mathbf{F}\| \in L^r_{\text{loc}}(U, \nu)$ is denoted by $L^r_{\text{loc}}(U, \mathbb{R}^d, \nu)$.\\
For simplicity, the following notations are used:
\begin{itemize}
\item $L^r(U) := L^r(U, dx)$,
\item $L^r_{\text{loc}}(U) := L^r_{\text{loc}}(U, dx)$,
\item $L^r(U, \mathbb{R}^d) := L^r(U, \mathbb{R}^d, dx)$,
\item $L^r_{\text{loc}}(U, \mathbb{R}^d) := L^r_{\text{loc}}(U, \mathbb{R}^d, dx)$.
\end{itemize}
The weak spatial derivative of a function $f$ with respect to the $i$-th coordinate is denoted by $\partial_i f$, provided it exists. The weak time derivative of $f$ is denoted by $\partial_t f$. Sobolev spaces are defined as follows:
\begin{itemize}
\item $H^{1,r}(U)$: The space of all functions $f \in L^r(U)$ with $\partial_i f \in L^r(U)$ for each $i = 1, \ldots, d$, equipped with the standard $H^{1,r}(U)$-norm
\item $H^{1,r}_{loc}(U)$: The set of all functions $f$ such that $f|_{W} \in H^{1,r}(W)$ for any bounded open set $W \subset \mathbb{R}^d$ with $\overline{W} \subset U$
%\item $H^{1,q}_0(U)$: The closure of $C_0^\infty(U)$ in $H^{1,q}(U)$ for $q \in [1, \infty)$.
%\item $H^{2,r}(U)$: The space of functions $f \in L^r(U)$ with $\partial_i f, \partial_i \partial_j f \in L^r(U)$ for all $i, j = 1, \ldots, d$, equipped with the $H^{2,r}(U)$-norm.
\end{itemize}
The weak Laplacian is defined as $\Delta f := \sum_{i=1}^d \partial_i \partial_i f$, and for a twice weakly differentiable function $f$, the weak Hessian matrix is $\nabla^2 f := (\partial_i \partial_j f)_{1 \leq i, j \leq d}$. Let $B = (b_{ij})_{1 \leq i, j \leq d}$ be a (possibly non-symmetric)
matrix of functions and define $\text{trace}(B) := \sum_{i=1}^d b_{ii}$. Therefore, by the commutativity of the differential operator, we obtain that for each twice weakly differentiable function $f$ and $\tilde{B}=(\tilde{b}_{ij})_{1 \leq i,j \leq d}:=\frac{1}{2}(B+B^T)$,
$$
\text{trace}({B}\nabla^2 f) = \sum_{i,j=1}^d b_{ij} \partial_i \partial_j f = \sum_{i,j=1}^d \tilde{b}_{ij} \partial_i \partial_j f =\text{trace}(\tilde{B}\nabla^2 f).
$$

\begin{defn} \label{defnofoker}
Let $x \in \mathbb{R}^d$, $T \in (0, \infty)$, and assume that {\bf (H-1)} holds. A family of probability measures $(\mu_t)_{t \in [0,T]}$ on $\mathcal{B}(\mathbb{R}^d)$ is called a solution to the Cauchy problem for the Fokker–Planck equation associated with $L$ and the initial distribution $\delta_x$ if the following properties (i)--(iii) are satisfied:
\begin{itemize}
\item[(i)] 
$\mu_0 = \delta_x$, where $\delta_x$ is the Dirac delta measure on $\mathcal{B}(\mathbb{R}^d)$ centered at $x$.
\item[(ii)] 
For each $f \in C_b(\mathbb{R}^d)$, the map $\ell_f: [0,T] \to \mathbb{R}$ defined by
$$
\ell_f(t) := \int_{\mathbb{R}^d} f(y) \mu_t(dy), \quad t \in [0,T],
$$
is continuous on $[0,T]$.
\item[(iii)]
$\mathbf{G} \in L^1(B \times (0,t), \mu_s ds)$ for any open ball $B$ and $t \in [0, T]$ and
\begin{equation} \label{fpkide1}
\int_{\mathbb{R}^d} \varphi \, d\mu_t = \varphi(x) + \int_0^t \int_{\mathbb{R}^d} L\varphi(y) \, \mu_s(dy) \, ds \quad \text{ for all $t \in [0, T]$ and $\varphi \in C_0^{\infty}(\mathbb{R}^d)$}.
\end{equation}
\end{itemize}
By \cite[Proposition 6.1.2]{BKRS15}, under the assumption that (i)--(ii) hold, \eqref{fpkide1} is equivalent to
$$
\iint_{\mathbb{R}^d \times (0, T)} \big(\partial_t \varphi + L \varphi\big)\, d\mu_s \, ds = 0 \quad \text{ for all $\varphi \in C_0^{\infty}(\mathbb{R}^d \times (0, T))$}.
$$
Additionally, a family of probability measures $(\mu_t)_{t \in [0, \infty)}$ on $\mathcal{B}(\mathbb{R}^d)$ is called a solution to the Cauchy problem for the Fokker–Planck equation associated with $L$ and the initial distribution $\delta_x$ if $(\mu_t)_{t\in [0, \infty)}$ satisfies (i)--(iii), where $[0,T]$ is replaced by $[0, \infty)$.
\end{defn}

\begin{defn} \label{martiprupto}
Let $x \in \mathbb{R}^d$, $T \in (0, \infty)$, and assume that {\bf (H-1)} holds. Let $\bar{\Omega}_T := C([0, T], \mathbb{R}^d)$ with the standard supremum norm on $[0, T]$. For each $t \in [0, T]$ and $\bar{\omega} \in \bar{\Omega}_T$, let $\bar{X}_t(\bar{\omega}) = \bar{\omega}(t)$.  
A probability measure $\bar{\mathbb{P}}_x$ on $(\bar{\Omega}_T, \mathcal{B}(\bar{\Omega}_T))$ is said to solve the martingale problem for $(L, C_0^{\infty}(\mathbb{R}^d))$ up to a finite time $T$ if the following properties hold:
\begin{itemize}
    \item[(i)] 
    $\bar{\mathbb{P}}_x \left( \bar{\omega} \in \bar{\Omega}_T : \bar{\omega}(0) = x \right) = 1$.
    \item[(ii)] 
    For each $f \in C_0^{\infty}(\mathbb{R}^d)$,
    \[
    \left( f(\bar{X}_{t}) - f(x) - \int_0^t Lf (\bar{X}_s) \, ds \right)_{t \in [0, T]}
    \]
    is a continuous martingale with respect to the measure $\bar{\mathbb{P}}_x$ and the natural filtration $\bar{\mathcal{F}}_t := \sigma (\bar{X}_s : s \in [0, t])$
    (the filtration 
   $(\bar{\mathcal{F}}_t)_{t \in [0,T]}$ and $\mathcal{B}(\bar{\Omega}_T)$
    are always considered with augmentation under the probability measure $\bar{\mathbb{P}}_x$).
\end{itemize}
\end{defn}
\centerline{}
\noindent
The following embedding result follows straightforwardly from Definition \ref{martiprupto}.  
\begin{prop} \label{prop12}
Let $x \in \mathbb{R}^d$, $T \in (0, \infty)$, and assume that {\bf (H-1)} holds. Assume that a probability measure $\bar{\mathbb{P}}_x$ on $(\bar{\Omega}_T, \mathcal{B}(\bar{\Omega}_T))$ solves the martingale problem for $(L, C_0^{\infty}(\mathbb{R}^d))$ up to a finite time $T$, as defined in Definition \ref{martiprupto}.  
Let $\tilde{\Omega} := C([0, \infty), \mathbb{R}^d)$, and for each $t \in [0, \infty)$ and $\tilde{\omega} \in \tilde{\Omega}$, define $\tilde{X}_t(\tilde{\omega}) := \tilde{\omega}(t)$.  
Let $\Theta: \bar{\Omega}_T \rightarrow \tilde{\Omega}$ be a map defined as follows: for each $\bar{\omega} \in \bar{\Omega}_T$, $\Theta(\bar{\omega})$ is a continuous function on $[0, \infty)$ satisfying  
\[
\Theta(\bar{\omega})(t) =
\begin{cases} 
\bar{\omega}(t), & t \in [0, T], \\ 
\bar{\omega}(T), & t \geq T.
\end{cases}
\]  
Define a probability measure $\tilde{\mathbb{P}}_x$ on $(\tilde{\Omega}, \mathcal{B}(\tilde{\Omega}))$ by  
\[
\tilde{\mathbb{P}}_x (\Theta(\Lambda)) := \bar{\mathbb{P}}_x(\Lambda) \quad \text{ for all } \Lambda \in \mathcal{B}(\bar{\Omega}_T).
\]  
Then, the following properties hold:  
\begin{itemize}
    \item[(i)] 
    $\tilde{\mathbb{P}}_x \left( \tilde{\omega} \in \tilde{\Omega} : \tilde{\omega}(0) = x \right) = 1$.
    \item[(ii)] 
    For each $f \in C_0^{\infty}(\mathbb{R}^d)$,  
    \[
    \left( f(\tilde{X}_{t\wedge T}) - f(x) - \int_0^{t \wedge T} Lf (\tilde{X}_s) \, ds \right)_{t \geq 0}
    \]
    is a continuous martingale with respect to the measure $\tilde{\mathbb{P}}_x$ and the natural filtration $\tilde{\mathcal{F}}_t := \sigma(\tilde{X}_s : s \in [0, t])$ (the filtration $(\tilde{\mathcal{F}}_t)_{t \geq 0}$ and $\mathcal{B}(\tilde{\Omega})$ are always considered with augmentation under the probability measure $\bar{\mathbb{P}}_x$).
\end{itemize}
\end{prop}

\section{Existence and ergodicity of solutions} \label{exisergo}

\begin{lem}
Let $(P_t)_{t>0}$ be a sub-Markovian semigroup as in \cite[Theorem 2.3.1]{LST22} (cf. \cite[Theorem 3.8]{LT21}). Additionally, define $P_0:=id$.
Let $B_r(z):= \{y \in \mathbb{R}^d: \|y-z\|<r \}$ be an open ball in $\mathbb{R}^d$ and fix $x_0 \in B_r(z)$. Let $(x_n, t_n)_{n \geq 1}$ be a sequence in $\mathbb{R}^d \times [0, \infty)$ such that $(x_n, t_n) \rightarrow  (x_0,0)$ as $n \rightarrow \infty$.
Then,
$$
\lim_{n \rightarrow \infty} P_{t_n}1_{B_r(z)}(x_n)=1.
$$
\end{lem}
\begin{proof}
Let $s:=\|x_0-z\|$. Then, $x_0 \in \overline{B}_s(z):=\{ x \in \mathbb{R}^d: \|x-z\| \leq s \} \subset B_r(z)$. Now, choose a compactly supported smooth function $f$ on $\mathbb{R}^d$ with $\text{supp}(f) \subset B_r(z)$, $0 \leq f \leq 1$ on $\mathbb{R}^d$ and $f \equiv 1$ on $\overline{B}_s(z)$. Then, $0 \leq f \leq 1_{B_r(z)}$ on $\mathbb{R}^d$, and hence
$$
0 \leq P_{t_n} f(x_n) \leq P_{t_n} 1_{B_r(z)}(x_n) \leq 1 \quad \text{ for any $n \geq 1$}.
$$
Since $\displaystyle \lim_{n \rightarrow \infty} P_{t_n} f(x_n) = f(x_0)=1$ by \cite[Lemma 2.30]{LST22} (cf. \cite[Proposition 3.6(iii)]{LT21}), we discover that
$\displaystyle \lim_{n \rightarrow \infty} P_{t_n}1_{B_r(z)}(x_n)=1$, as desired.
\end{proof}

\begin{prop} \label{continuprop}
Under the assumption {\bf (H-1)}, 
let $(P_t)_{t>0}$ be a sub-Markovian semigroup as in \cite[Theorem 2.3.1]{LST22} (cf. \cite[Theorem 3.8]{LT21}) and additionally define $P_0:=id$.
Let $(P_t(x,dy))_{t>0}$ be a family of sub-probability measures on $\mathcal{B}(\mathbb{R}^d)$ as in \cite[Theorem 3.1]{LST22} (cf. \cite[Proposition 3.10]{LT21}) and additionally define $P_0(x,dy):=\delta_x$. Let $x_0 \in \mathbb{R}^d$ be fixed and $f \in \mathcal{B}_b(\mathbb{R}^d)$ be continuous at $x_0$. Define $v_f: \mathbb{R}^d \times [0, \infty) \rightarrow \mathbb{R}$ given by
$$
v_f(x,t):= P_t f (x) = \int_{\mathbb{R}^d} f(y) P_t(x, dy) \quad \text{ for any $(x,t) \in \mathbb{R}^d \times [0, \infty)$}.
$$
Then, for each $t_0 \in [0, \infty)$, $v_f$ is continuous at $(x_0, t_0)$.
\end{prop}
\begin{proof}
By \cite[Theorem 2.3.1]{LST22} (cf. \cite[Theorem 3.8]{LT21}), it is enough to show that $v_f$ is continuous at $(x_0, 0)$. Let $(x_n, t_n)_{n \geq 1}$ be a sequence in $\mathbb{R}^d \times [0, \infty)$ such that $(x_n, t_n) \rightarrow (x_0, 0)$ as $n \rightarrow \infty$. Now, let $\varepsilon>0$ be given.
Since $f$ is continuous at $x_0$, there exists $\delta>0$ such that $|f(y)-f(x_0)|<\varepsilon$ for any $y \in B_{\delta}(x_0)$. Then, by the sub-Markovian property of $(P_t)_{t>0}$, we have
\begin{align*}
&|P_{t_n} f(x_n)-f(x_0) | \leq\left| \int_{\mathbb{R}^d} f(y) P_{t_n}(x_n, dy) -\int_{\mathbb{R}^d} f(x_0) P_{t_n}(x_n, dy) 		\right|+\Big |f(x_0) P_{t_n}(x_n, \mathbb{R}^d)-f(x_0)\Big|   \\
& \leq \left| \int_{\mathbb{R}^d \setminus B_{\delta}(x_0)}| f(y)-f(x_0) |P_{t_n}(x_n, dy) \right| + \left| \int_{ B_{\delta}(x_0)}| f(y)-f(x_0) |P_{t_n}(x_n, dy) \right| +|f(x_0)| (1-P_{t_n}1_{\mathbb{R}^d}(x_n)) \\
& \leq 2 \|f\|_{L^{\infty}(\mathbb{R}^d)} P_{t_n}(x_n,\mathbb{R}^d \setminus B_{\delta}(x_0)) + \varepsilon P_{t_n}(x_n, B_{\delta}(x_0)) +\|f\|_{L^{\infty}(\mathbb{R}^d)} (1-P_{t_n}1_{B_{\delta}(x_0)}(x_n)) \\
& \leq2 \|f\|_{L^{\infty}(\mathbb{R}^d)} \left( 1-P_{t_n}1_{B_{\delta}(x_0)}(x_n)\right) + \varepsilon+\|f\|_{L^{\infty}(\mathbb{R}^d)} (1-P_{t_n}1_{B_{\delta}(x_0)}(x_n)).
\end{align*}
Therefore, by Proposition \ref{continuprop}, $\limsup_{n \rightarrow \infty} |P_{t_n} f(x_n)-f(x_0) | \leq \varepsilon $, and since $\varepsilon>0$ is arbitrarily chosen, the assertion follows.
\end{proof}

\begin{theo} \label{mainexisthm}
Let $x \in \mathbb{R}^d$, and assume that {\bf (H-1)} holds. 
Let $(P_t)_{t>0}$ be a sub-Markovian semigroup as in \cite[Theorem 2.3.1]{LST22} (cf. \cite[Theorem 3.8]{LT21}) and  $(P_t(x,dy))_{t>0}$ be a family of sub-probability measures on $\mathcal{B}(\mathbb{R}^d)$ as in \cite[Theorem 3.1]{LST22} (cf. \cite[Proposition 3.10]{LT21}). Define $P_0(x, dy) := \delta_x$. Then, the following holds:
\begin{itemize}
\item[(i)]
$\mathbf{G} \in L^1(B \times (0,T), P_t(x, dy) dt)$ for any open ball $B$ and $T \in (0, \infty)$, and
\begin{equation*} \label{fpkidesemi}
\int_{\mathbb{R}^d} \varphi(y) \, P_t(x,dy) = \varphi(x) + \int_0^t \int_{\mathbb{R}^d} L\varphi(y) \, P_s(x, dy) \, ds \quad \text{for all $t \in [0, \infty)$ and $\varphi \in C_0^{\infty}(\mathbb{R}^d)$}.
\end{equation*}
\item[(ii)]
There exists $\mu=\rho dx$ with $ \rho \in H^{1,p}_{loc}(\mathbb{R}^d) \cap C(\mathbb{R}^d)$ and $\rho(x)>0$ for any $x \in \mathbb{R}^d$
such that
\begin{equation} \label{infiniesim}
\int_{\mathbb{R}^d} L \varphi \,d\mu =0 \quad \text{ for all $\varphi \in C_0^{\infty}(\mathbb{R}^d)$}.
\end{equation}
Moreover, for each $t \in (0, \infty)$, there exists $p_t(x, \cdot) \in L^{\infty}(\mathbb{R}^d)$ such that
$$
P_t(x, dy) = p_t(x,y) \rho(y) dy. \quad 
$$
\item[(iii)] $P_t(x, E)>0$ for any $t \in (0, \infty)$ and $E \in \mathcal{B}(\mathbb{R}^d)$ with $dx(E)>0$.
\item[(iv)] For each $f \in C_b(\mathbb{R}^d)$, the map $\ell_f: [0,\infty) \to \mathbb{R}$ defined by
$$
\ell_f(t) := \int_{\mathbb{R}^d} f(y) P_t(x,dy), \quad t \in [0,\infty),
$$
is continuous on $[0,\infty)$.
Moreover, assume that $(P_t)_{t>0}$ is conservative, i.e. $P_t 1_{\mathbb{R}^d} = 1$ on $\mathbb{R}^d$ for any $t>0$ (cf. Remark \ref{growthcondi}(i)). Then,
$(P_t(x, dy))_{t \geq 0}$ is a family of probability measures on $\mathcal{B}(\mathbb{R}^d)$, which is a solution to the Fokker-Planck equations associated with $L$ and the initial distribution $\delta_x$ as in Definition \ref{defnofoker}.
\item[(v)] Assume that $\mu$ in (ii) is finite and an invariant measure for $(P_t)_{t>0}$ (cf. Remark \ref{growthcondi}(ii)), i.e.
$$
\mu(E)=\int_{\mathbb{R}^d}P_t 1_E \,d\mu \quad \text{ for any $E \in \mathcal{B}(\mathbb{R}^d)$}.
$$
Then, $(P_t)_{t>0}$ is conservative and for each $E \in \mathcal{B}(\mathbb{R}^d)$
\begin{equation} \label{longtimeinf}
\lim_{t \rightarrow \infty} P_t(x, E) = \frac{\mu(E)}{\mu(\mathbb{R}^d)}.
\end{equation}
\item[(vi)] Assume that $\mu$ in (ii) is finite and $(P_t)_{t>0}$ is conservative. Then,  for each $E \in \mathcal{B}(\mathbb{R}^d)$, \eqref{longtimeinf} holds.
\end{itemize}
\end{theo}
\begin{proof}
(i) First, let $B$ be an open ball in $\mathbb{R}^d$ and $T \in (0, \infty)$. Then, by \cite[Proposition 3.13]{LST22},
\begin{align*}
\int_0^T \int_{B} \|\mathbf{G}(y)\| P_t(x,dy) dt &= \int_0^T \int_{\mathbb{R}^d} \|1_{B}(y)\mathbf{G}(y)\| P_t(x,dy) dt  \\
&= \int_0^T P_t (1_{B} \|\mathbf{G}\|) (x)dt \leq c_{x,p} e^T\| \mathbf{G} \|_{L^p(B, \mu)},
\end{align*}
where $c_{x,p}>0$ is a constant which only depends on $x$ and $p$.
%observe that by Lemma, for each $f \in C_b(\mathbb{R}^d)$, the map $t \mapsto \int_{\mathbb{R}^d} f(y) P_t(x,dy)$ is continuous on $[0, \infty)$.
Moreover, it follows from \cite[Proposition 3.13(iii)]{LST22} (cf. \cite[Lemma 3.14(iii)]{LT21}) that,  for each $t \in [0,\infty)$ and $\varphi \in C_0^{\infty}(\mathbb{R}^d)$, we have
$$
\int_{\mathbb{R}^d} \varphi(y)P_t(x,dy) - \varphi(x)=P_t \varphi(x) - \varphi(x) = \int_{0}^t P_s L \varphi (x)ds = \int_0^t \int_{\mathbb{R}^d} L \varphi (y) P_s(x, dy) ds,
$$
and hence (i) follows. \\ \\
(ii) Indeed, $(P_t)_{t>0}$ is a sub-Markovian $C_0$-semigroup of contractions on $L^1(\mathbb{R}^d, \mu)$, where $\mu=\rho dx$ with $ \rho \in H^{1,p}_{loc}(\mathbb{R}^d) \cap C(\mathbb{R}^d)$ and $\rho(x)>0$ for all $x \in \mathbb{R}^d$ and $\mu$ satisfies \eqref{infiniesim} (see \cite[Section 2.2]{LST22}). In particular, $P_t(x,dy) \ll \mu$, and hence there exists $p_t(x, \cdot) \in L^1(\mathbb{R}^d, \mu)$ such that $P_t(x,dy)=p_t(x, y) \mu(dy)$.
Moreover, for any $f \in L^1(\mathbb{R}^d, \mu)$ and $t>0$
$$
\left| \int_{\mathbb{R}^d} f(y) p_t(x, y) \mu(dy)  \right |=\left|\int_{\mathbb{R}^d}f(y) P_t(x, dy) \right|=|P_t f(x)| 
\leq K_{t,x} \|f\|_{L^1(\mathbb{R}^d, \mu)}.
$$
Thus, from the Riesz representation theorem, $p_t(x, \cdot) \in L^{\infty}(\mathbb{R}^d, \mu)$ and $\|p_t(x, \cdot) \|_{L^{\infty}(\mathbb{R}^d, \mu)} \leq K_{t,x}$. \\ \\
(iii) It follows from \cite[Proposition 2.39(i)]{LST22} (cf. \cite[Corollary 4.8(ii)]{LT21}). \\ \\
(iv) Let $f \in C_b(\mathbb{R}^d)$. The continuity of $\ell_f$ on $(0, \infty)$ follows from \cite[Theorem 2.3.1]{LST22} (cf. \cite[Theorem 3.8]{LT21}). The continuity of $\ell_f$ at $0$ holds by Proposion \ref{continuprop}. If $(P_t)_{t>0}$ is conservative, then $P_t(x, \mathbb{R}^d)=P_t1_{\mathbb{R}^d}(x)=1$ for all $t>0$, and hence the assertion follows from Theorem \ref{mainexisthm}(i).
\\ \\
(v) It holds by \cite[Theorem 3.46(iv)]{LST22}.\\  \\
(vi) Since $\mu$ is finite, the conservativeness of $(P_t)_{t>0}$ implies that $\mu$ is an invariant measure for $(P_t)_{t>0}$ by \cite[Remark 2.13]{LST22}. Thus, the assertion follows from (v).
\end{proof}

\begin{rem} \label{growthcondi}
\begin{itemize}
\item[(i)]
If there exist constants $M>0$ and $N_0 \in \mathbb{N}$ and a function $g\in C^2(\mathbb{R}^d \setminus \overline{B}_{N_0}) \cap C(\mathbb{R}^d)$ such that
\begin{equation} \label{lapunoconser}
 L g\leq M g \quad \text{ for a.e. $\mathbb{R}^d\setminus \overline{B}_{N_0}$} \quad \text{ and } \quad \lim_{\|x\| \rightarrow \infty} g(x) = \infty,
\end{equation}
then it follows from \cite[Lemma 3.26, Corollary 3.23]{LST22} that $(P_t)_{t>0}$ in Theorem \ref{mainexisthm} is conservative. As a typical example for $g$ satisfying the above conditions, one can choose a function $g(x):= \ln ( \|x\|^2 \vee N_0^2  )+2$, $x \in \mathbb{R}^d$. Thus, if
\begin{equation} \label{sprincons}
-\frac{\langle A(x)x, x \rangle}{\|x\|^2} + \frac{1}{2} \mathrm{trace} A(x) + \langle \mathbf{G}(x), x \rangle \leq M \|x\|^2 \left(\ln \|x\| + 1 \right) \quad \text{ for a.e. $\mathbb{R}^d\setminus \overline{B}_{N_0}$},
\end{equation}
then \eqref{lapunoconser} is satisfied, so that $(P_t)_{t>0}$ is conservative. It can be directly verified that if {\bf (H-2)} holds, then \eqref{sprincons} is fulfilled.

\item[(ii)]
If there exist constants $M>0$ and $N_0 \in \mathbb{N}$ and a function $g\in C^2(\mathbb{R}^d \setminus \overline{B}_{N_0}) \cap C(\mathbb{R}^d)$ such that
\begin{equation} \label{lapunocfinin}
 L g\leq -M \quad \text{ for a.e. $\mathbb{R}^d\setminus \overline{B}_{N_0}$} \quad \text{ and } \quad \lim_{\|x\| \rightarrow \infty} g (x)= \infty,
\end{equation}
then it follows from \cite[Lemma 3.26]{LST22} that $\mu$ in Theorem \ref{mainexisthm}(iii) is finite and an invariant measure for $(P_t)_{t>0}$.
As a typical example for $g$ satisfying above, one can also choose a function $g(x):= \frac12 \ln ( \|x\|^2 \vee N_0^2  )+1$, $x \in \mathbb{R}^d$, or alternatively,
$g(x)=\frac12 \|x\|^2$, $x \in \mathbb{R}^d$. Then, if either
\begin{equation} \label{sprininva1}
-\frac{\langle A(x)x, x \rangle}{\|x\|^2} + \frac{1}{2} \mathrm{trace} A(x) + \langle \mathbf{G}(x), x \rangle \leq -M \|x\|^2 \quad \text{ for a.e. $\mathbb{R}^d\setminus B_{N_0}$},
\end{equation}
or
\begin{equation} \label{sprininva2}
\frac{1}{2} \mathrm{trace} A(x) + \langle \mathbf{G}(x), x \rangle \leq -M \quad \text{ for a.e. $\mathbb{R}^d\setminus B_{N_0}$},
\end{equation}
holds, then \eqref{lapunocfinin} is satisfied, so that $\mu$ in Theorem \ref{mainexisthm}(iii) is finite and an invariant measure for $(P_t)_{t>0}$.
\end{itemize}
\end{rem}
\centerline{}
\noindent
The following result presents a peculiar property related to the growth of diffusion coefficients when $d=2$.
\begin{prop} \label{dim2growthm}
Under the assumption {\bf (H-1)} with $d=2$, write
$
A:=\begin{pmatrix}
a_{11} & a_{12} \\
a_{12} & a_{22}
\end{pmatrix}
$.
Let $(P_t)_{t>0}$, $(P_t(x,dy))_{t \geq 0}$ and $\mu=\rho dx$ be as in Theorem \ref{mainexisthm}.
Then, the following hold:
\begin{itemize}
\item[(i)]
If there exist constants $M>0$ and $N_0 \in \mathbb{N}$ such that
\begin{equation*} %\label{sprinconsdim2}
\frac{|a_{11}-a_{22}|}{2} +|a_{12}|+ \langle \mathbf{G}(x), x \rangle \leq M \|x\|^2 \left(\ln \|x\| + 1 \right) \quad \text{ for a.e. $\mathbb{R}^d\setminus \overline{B}_{N_0}$},
\end{equation*}
then $(P_t)_{t>0}$ as in Theorem \ref{mainexisthm} is conservative, and hence by Theorem \ref{mainexisthm}(iv), for each $x \in \mathbb{R}^d$, $(P_t(x, dy))_{t \geq 0}$ as in Theorem \ref{mainexisthm} is a solution to the Fokker-Planck equations for $L$ with the initial distribution $\delta_x$ as in Definition \ref{defnofoker}.
\item[(ii)]
If there exist constants $M>0$ and $N_0 \in \mathbb{N}$ such that
\begin{equation*} 
\frac{|a_{11}-a_{22}|}{2} +|a_{12}|+ \langle \mathbf{G}(x), x \rangle \leq -M \|x\|^2 \quad \text{ for a.e. $\mathbb{R}^d\setminus B_{N_0}$},
\end{equation*}
then $\mu$ in Theorem \ref{mainexisthm} is finite and an invariant measure for $(P_t)_{t>0}$, so that for each $x \in \mathbb{R}^d$ and $E \in \mathcal{B}(\mathbb{R}^d)$
$$
\lim_{t \rightarrow \infty} P_t(x, E) = \frac{\mu(E)}{\mu(\mathbb{R}^d)}.
$$
\end{itemize}
\end{prop}
\begin{proof}
(i) Since $A$ is symmetric and strictly positive definite, by elementary linear algebra, there exist a diagonal matrix of functions $D=\begin{pmatrix}
\Psi_1 & 0 \\
0 & \Psi_2
\end{pmatrix}$ and orthonormal matrix of functions $Q$ such that
$$
A = Q^TDQ \quad \text{ for all $x \in \mathbb{R}^d$}.
$$
Then, by the proof of \cite[Corollary 3.28]{LST22},
\begin{equation} \label{growthdif}
-\frac{\langle A(x)x, x \rangle}{\|x\|^2} + \frac{1}{2} \text{trace}A(x) \leq  \frac{|\Psi_1(x) - \Psi_2(x)|}{2} \quad \text{ for all $x \in \mathbb{R}^d$}.
\end{equation}
Meanwhile, from the elementary linear algebra, 
$$
\Psi_1+\Psi_2=\text{trace}D=\text{trace}A=a_{11}+a_{22} \quad \text{ on $\mathbb{R}^d$}
$$ 
and
$$
\Psi_1 \Psi_2 =\det D = \det A = a_{11} a_{22}-a_{12}^2  \quad \text{ on $\mathbb{R}^d$}.
$$
Thus, 
\begin{align}
| \Psi_1 - \Psi_2 | &= \sqrt{(\Psi_1 + \Psi_2)^2-4\Psi_1 \Psi_2} =\sqrt{ (a_{11} +a_{22})^2 -4(a_{11} a_{22} -a^2_{12}) } \nonumber \\
&=\sqrt{(a_{11}-a_{22})^2 +4 a_{12}^2} \leq |a_{11}-a_{22}| + 2|a_{12}| \quad \text{ on $\mathbb{R}^d$} \label{growth2dim}
\end{align}
Therefore, the assertion follows from \eqref{growth2dim}, \eqref{growthdif} and \eqref{sprincons}.\\ \\
(ii) Analogously, the assertion follows from \eqref{growth2dim}, \eqref{growthdif} and \eqref{sprininva1}.

\end{proof}

\begin{rem} \label{growthconrema}
As observed in the above Proposition \ref{dim2growthm}, in dimension $d = 2$, if the difference between the two diagonal coefficients of the diffusion matrix is bounded by a quadratic times logarithmic growth, the existence of a solution to the Fokker-Planck equation is guaranteed through Theorem 3.3(iv), no matter how large the growth of the individual diagonal coefficients is. This demonstrates that the existence of solutions can be discussed under much weaker growth conditions for the coefficients than those required by the condition {\bf (H-2)} in Theorem \ref{mainresu}(ii), which ensures the uniqueness of solutions to the Fokker-Planck equation. 
In fact, in the case of \eqref{sprininva1}, it can be checked that if the growth condition of $\langle \mathbf{G}(x), x \rangle$ is sufficiently negative, the diffusion coefficients can exhibit rapid growth. However, such rapid growth is not permissible under the condition {\bf (H-2)}.
Therefore, further research may be needed to generalize the growth conditions of {\bf(H-2)} to broaden the applicability of the superposition principle for solutions to the Cauchy problem for the Fokker–Planck equation associated with $L$ and the initial distribution $\delta_x$ as in Definition \ref{defnofoker} (cf. \cite[Question 8]{BRS23}).
\end{rem}

\section{Uniqueness of solutions by stochastic analysis} \label{unibystoc}
\begin{theo} \label{cholesky}
Let $A=(a_{ij})_{1 \leq i,j \leq d}$ be a symmetric matrix of functions as in {\bf (H-1)}. Then, there exists a (unique) lower triangular matrix of functions $\sigma=(\sigma_{ij})_{1 \leq i,j \leq d}$ with $\sigma_{ij} \in H^{1,p}_{loc}(\mathbb{R}^d) \cap C(\mathbb{R}^d)$ and $\sigma_{ii}>0$ for all $1 \leq i,j \leq d$ such that
$$
A=\sigma \sigma^T \quad \text{ on $\mathbb{R}^d$}
$$
\end{theo}
\begin{proof}
The existence and uniqueness of a lower triangular matrix of functions $\sigma=(\sigma_{ij})_{1 \leq i,j \leq d}$ such that $A=\sigma \sigma^T$ and $\sigma_{ii}>0$ for all $1 \leq i\ \leq d$ is shown as in \cite[Theorem 6.3.1]{AK08}. Now, it is enough to show that $\sigma_{ij} \in H^{1,p}_{loc}(\mathbb{R}^d) \cap C(\mathbb{R}^d)$ for all $1 \leq i,j \leq d$. Indeed, the specific calculation for the components of the lower triangular matrix of functions $\sigma$ are described as in \cite[Section 6.3.1]{AK08}. These are conducted by the columnwise computation and described as the following {\bf Algorithm 1}.
\begin{algorithm}
\caption{Cholesky Decomposition (Columnwise computation)}
\begin{algorithmic}[1]
\Require Symmetric positive-definite matrix of functions $A$ of size $d \times d$
\Ensure Lower triangular matrix of functions $\sigma$ such that $A = \sigma \sigma^T$ and the diagonal components of $\sigma$ are strictly positive.

\State Compute the first column ($j=1$):
\State $\sigma_{11} = \sqrt{a_{11}}$
\For{$i = 2$ to $d$}
    \State $\sigma_{i1} = \frac{a_{i1}}{\sigma_{11}}$
\EndFor

\For{$j = 2$ to $d$}
    \State Compute the diagonal element of column $j$:
    \begin{equation*}
        \sigma_{jj} = \sqrt{a_{jj} - \sum_{k=1}^{j-1} \sigma_{jk}^2}
    \end{equation*}
    \State Compute the off-diagonal elements of column $j$:
    \For{$i = j+1$ to $d$}
        \begin{equation*}
            \sigma_{ij} = \frac{1}{\sigma_{jj}} \left(a_{ij} - \displaystyle \sum_{k=1}^{j-1} \sigma_{jk}\sigma_{ik} \right)
        \end{equation*}
    \EndFor
\EndFor
\end{algorithmic}
\end{algorithm}
We now show that by using the {\bf Algorithm 1} $\sigma_{ij} \in H^{1,p}_{loc}(\mathbb{R}^d) \cap C(\mathbb{R}^d)$ for all $1 \leq i,j \leq d$. First, the step 1-5 in {\bf Algorithm 1} deduce that $\sigma_{i1} \in H^{1,p}_{loc} (\mathbb{R}^d) \cap C(\mathbb{R}^d)$ for all $1 \leq i \leq d$, since $a_{11} \in H^{1,p}_{loc}(\mathbb{R}^d) \cap C(\mathbb{R}^d)$ and $a_{11}(x)>0$ for all $x \in \mathbb{R}^d$. Next, suppose that there exists a column index $m$, $2 \leq m \leq d$, such that 
$\sigma_{ik} \in H^{1,p}_{loc}(\mathbb{R}^d) \cap C(\mathbb{R}^d)$  for all $1 \leq k \leq m-1$ and
$1 \leq i \leq d$. Then, by the step 7 in {\bf Algorithm 1}, 
$$
\sigma_{mm} = \sqrt{a_{mm} - \sum_{k=1}^{m-1} \sigma_{mk}^2} \in H^{1,p}_{loc}(\mathbb{R}^d) \cap C(\mathbb{R}^d)
$$
and by the step 9--10 in {\bf Algorithm 1},  
$$
\sigma_{im} = \frac{1}{\sigma_{mm}} \left(a_{im} -\sum_{k=1}^{m-1} \sigma_{mk}\sigma_{ik} \right) \in H^{1,p}_{loc}(\mathbb{R}^d) \cap C(\mathbb{R}^d) \quad \text{ for all $m+1 \leq i \leq d$}. 
$$
Since $\sigma_{im}=0$ for all $1 \leq i \leq m-1$, we obtain that
$\sigma_{im} \in H^{1,p}_{loc}(\mathbb{R}^d) \cap C(\mathbb{R}^d)$ for all $1 \leq i \leq d$.
By induction, we conclude that $\sigma_{ij} \in H^{1,p}_{loc}(\mathbb{R}^d) \cap C(\mathbb{R}^d)$ for all $1 \leq i \leq d$ and $1 \leq j \leq d$.
\end{proof}
\centerline{}
\noindent
The following result demonstrates that a solution to the martingale problem up to a finite time $T$ can be identified as a weak solution to an It\^{o}-SDE.
\begin{theo}  \label{martiweak} \rm
Let $x \in \mathbb{R}^d$, $T \in (0, \infty)$, and assume that {\bf (H-1)} holds. Let $\bar{\mathbb{P}}_x$ be a solution to the martingale problem for $(L, C_0^{\infty}(\mathbb{R}^d))$ up to a finite time $T$ as in Definition \ref{martiprupto}. 
 Let $(\tilde{\Omega}, \tilde{\mathcal{F}},  (\tilde{\mathcal{F}})_{t \geq 0}, \tilde{\mathbb{P}}_x)$ be a filtered probability space and $(\tilde{X}_t)_{t \geq 0}$ be a canonical process as in Proposition \ref{prop12} such that (i)--(ii) in Proposition \ref{prop12} hold.
  Let $\sigma$ be a lower triangular matrix of functions such that $A=\sigma \sigma^T$, $\sigma_{ii}>0$ for all $1 \leq i \leq d$ and $\sigma_{ij} \in H^{1,p}_{loc}(\mathbb{R}^d) \cap C(\mathbb{R}^d)$ for all $ 1\leq i, j\leq d$, as in Theorem \ref{cholesky}. Then, there exists an extension $(\hat{\Omega}, \hat{\mathcal{F}}, \hat{\mathbb{P}}_x, 
 (\hat{\mathcal{F}_t})_{t \geq 0}, (\hat{X}_t)_{t \geq 0})$ of $(\tilde{\Omega}, \tilde{\mathcal{F}}, \tilde{\mathbb{P}}_x, (\tilde{\mathcal{F}_t})_{t \geq 0}, (\tilde{X}_t)_{t \geq 0})$, and there exists an $(\hat{\mathcal{F}}_t)_{t \geq 0}$-standard Brownian motion $(\hat{W}_t)_{t \geq 0}$
such that 
$$
\hat{X}_t = x+ \int_0^t \sigma(\hat{X}_s) d \hat{W}_s + \int_0^t \mathbf{G}(\hat{X}_s) ds, \quad 0 \leq t \leq T, \qquad \hat{\mathbb{P}}_x\text{-a.s.}
$$
In particular, for each $0 \leq t \leq T$, we have that
$$
\bar{\mathbb{P}}_x \circ \bar{X}^{-1}_{t} =
\tilde{\mathbb{P}}_x \circ \tilde{X}^{-1}_{t} = \hat{\mathbb{P}} \circ \hat{X}^{-1}_{t}  \qquad \text{ on \, $\mathcal{B}(\mathbb{R}^d)$}.
$$
\end{theo}
\noindent
\begin{proof}
Let $v \in C^2(\mathbb{R}^d)$ and define
$$
M_t^v:=v(\tilde{X}_{t \wedge T})-v(x) - \int_0^{t \wedge T} Lv (\tilde{X}_s)ds, \quad 0 \leq t< \infty.
$$
Then, by Proposition \ref{prop12} and a simple extension to smooth functions with compact support, we obtain that
$(M^v_t)_{t \geq 0}$ is a local $(\tilde{\mathcal{F}}_t)_{t \geq 0}$-martingale with respect to $\tilde{\mathbb{P}}_x$. Using It\^{o}'s formula for the semimartinagle $\big(v(\tilde{X}_{t \wedge T})\big)_{t \geq 0}$, we have
$$
v(\tilde{X}_{t \wedge T})^2 -v(x)^2 = \int_0^t 2v(\tilde{X}_{s \wedge T}) dM^v_s + \int_0^t 2v Lv (\tilde{X}_{s \wedge T}) ds+ \langle  M^v \rangle_t \quad \text{ for all $t \geq 0$}.
$$
In addition, $(M_t^{v^2})_{t \geq 0}$ is a local $(\mathcal{F}_t)_{t \geq 0}$-martingale and it holds that
\begin{align*}
M_t^{v^2}&=v(\tilde{X}_{t \wedge T})^2-v(x)^2 - \int_0^{t \wedge T} Lv^2 (\tilde{X}_s)ds \\
& =v(\tilde{X}_{t \wedge T})^2-v(x)^2 - \int_0^{t \wedge T} 2vLv (\tilde{X}_s)ds - \int_0^{t \wedge T} \langle A \nabla v, \nabla v \rangle (\tilde{X}_s)ds \quad \text{ for all $t \geq 0$}.
\end{align*}
Therefore, we get 
$
\left( \langle M^v \rangle_t - \int_0^{t \wedge T} \langle A \nabla v, \nabla v \rangle(\tilde{X}_s) ds \right)_{t \geq 0}
$
is a continuous $(\tilde{\mathcal{F}}_t)_{t \geq 0}$-martingale starting at $0$ with bounded variation. Therefore,
$\langle M^v \rangle_t = \int_0^{t \wedge T} \langle A \nabla v, \nabla v \rangle (\tilde{X}_s) ds$\, for all $t \geq 0$.
Hence, by the polarization identity, for each $u,v \in C^2(\mathbb{R}^d)$ it holds that
\begin{align*}
\langle M^u, M^v \rangle_t&=\frac{1}{2} \left( \langle M^u + M^v, M^u + M^v \rangle_t - \langle M^u, M^u \rangle_t - \langle M^v, M^v \rangle_t \right) \\
&= \frac{1}{2} \left( \langle M^{u+v}, M^{u+v} \rangle_t - \langle M^u, M^u \rangle_t - \langle M^v, M^v \rangle_t \right) \\
& = \frac{1}{2} \left( \int_{0}^{t\wedge T} \langle A \nabla (u+v), \nabla (u+v) \rangle (\tilde{X}_s)ds - \int_0^{t \wedge T} \langle A \nabla u, \nabla u \rangle (\tilde{X}_s) ds - \int_{0}^{t \wedge T} \langle A \nabla v, \nabla v \rangle (\tilde{X}_s)	ds \right) \\
& = \int_{0}^{t \wedge T} \langle A \nabla u, \nabla v \rangle (\tilde{X}_s) ds =\int_{0}^{t} 1_{[0,T]}(s)\langle A \nabla u, \nabla v \rangle (\tilde{X}_s) ds \quad \text{ for all $t \geq 0$}.
\end{align*}
For each $i=1, \ldots, d$, let $u_i:=x_i$ be the $i$-th projection coordinate function. And for each $1 \leq i,j \leq d$,
define
$$
\Phi_{ij}(s):=1_{[0,T]}(s)\langle A \nabla u_i, \nabla u_j \rangle(\tilde{X}_s)=1_{[0,T]}(s) a_{ij}(\tilde{X}_s), \quad 0 \leq s< \infty,
$$
and
$$
\Psi_{ij}(s) := 1_{[0,T]}(s) \sigma_{ij}(\tilde{X}_s), \quad 0 \leq s< \infty.
$$
Then, we obtain that for each $1 \leq i,j \leq d$,
$\Phi_{ij}(s) = \sum_{k=1}^d\Psi_{ik}(s) \Psi_{jk}(s)$, for all $s \in [0, \infty)$. Moreover, we find that
$$
\int_{0}^t |\Phi(s)| ds < \infty, \;\quad \int_{0}^t |\Psi(s)|^2 ds <\infty \quad \text{ for all $t>0$}, \quad \text{$\tilde{\mathbb{P}}_x$-a.s.}
$$
and
$$
\langle M^{u_i}, M^{u_j} \rangle_t   =\int_{0}^t \Phi_{ij}(s) ds \quad \text{ for all $ t\geq 0$},\; \quad \text{$\tilde{\mathbb{P}}_x$-a.s.}
$$
According to the Ikeda-Watanabe theorem (\cite[Chapter II, Theorem 7.1$^\prime$]{IW89}), there exists an extension $(\hat{\Omega}, \hat{\mathcal{F}}, \hat{\mathbb{P}}_x, 
 (\hat{\mathcal{F}_t})_{t \geq 0}, (\hat{X}_t)_{t \geq 0})$ of $(\tilde{\Omega}, \tilde{\mathcal{F}}, \tilde{\mathbb{P}}_x, (\tilde{\mathcal{F}_t})_{t \geq 0}, (\tilde{X}_t)_{t \geq 0})$, and there exists an $(\hat{\mathcal{F}}_t)_{t \geq 0}$-standard Brownian motion $(\hat{W}_t)_{t \geq 0}=\big(\hat{W}^1_t, \ldots, \hat{W}^d_t \big)_{t \geq 0}$
such that for each $i=1, \ldots, d$,
$$
\hat{X}^i_{t \wedge T} -x_i - \int_{0}^{t \wedge T} g_i(\hat{X}_s) ds=M^{u_i}_t= \sum_{k=1}^d \int_0^t \Psi_{ik} (s) \,d\hat{W}^k_s=\sum_{k=1}^d \int_0^{t \wedge T} \sigma_{ik} (\hat{X}_s) \,d\hat{W}^k_s, \quad \forall t>0, \quad \text{$\hat{\mathbb{P}}_x$-a.s.}
$$
as desired.
\end{proof}
\centerline{}
\noindent
The following is a direct consequence of \cite[Theorem 1.1]{Z11} and straightforward localization arguments.
\begin{theo}  \label{pathwiseuni} \rm
Let $\mathbf{G} \in L^p_{loc}(\mathbb{R}^d, \mathbb{R}^d)$, and let $\tilde{\sigma} = (\tilde{\sigma}_{ij})_{1 \leq i,j \leq d}$ be a matrix of functions with $\tilde{\sigma}_{ij} \in H^{1,p}_{loc}(\mathbb{R}^d) \cap C(\mathbb{R}^d)$ such that $\tilde{\sigma} \tilde{\sigma}^T$ is locally uniformly strictly elliptic on $\mathbb{R}^d$. That is, for any open ball $B$, there exist strictly positive constants $\lambda_B$ and $\Lambda_B$ such that
$$
\lambda_{B} \| \xi \|^2 \leq \langle \tilde{\sigma}(x) \tilde{\sigma}(x)^T \xi, \xi \rangle \leq \Lambda_{B} \| \xi \|^2 \quad \text{ for all $\xi \in \mathbb{R}^d$ and $x \in B$}.
$$
Let $(\tilde{\Omega}, \tilde{\mathcal{F}}, (\tilde{\mathcal{F}}_t)_{t \geq 0}, \tilde{\mathbb{P}}_x)$ be a filtered probability space, and let
$(\tilde{W}_t)_{t \geq 0}$ be a standard $(\tilde{\mathcal{F}}_t)_{t \geq 0}$-Brownian motion. 
For each $i \in \{1, 2 \}$, let $(\tilde{X}^i_t)_{t \geq 0}$ be a stochastic process adapted to $(\tilde{\mathcal{F}}_t)_{t \geq 0}$ such that
$$
\tilde{X}^i_t = x +  \int_0^t \tilde{\sigma}(\tilde{X}^i_s) \, d \tilde{W}_s   +  \int_0^t \mathbf{G} (\tilde{X}^i_s) \, ds, \quad 0 \leq t \leq T, \quad \text{$\tilde{\mathbb{P}}_x$-a.s.}
$$
Then, 
$$
\tilde{\mathbb{P}}_x  \left(  \tilde{X}^1_t = \tilde{X}_t^2 \quad \text{for all $t \in [0,T]$}	\right) = 1.
$$
\end{theo}
\begin{proof}
Let $N \in \mathbb{N}$ be such that $x \in B_N$. For each $n \geq N+1$, let $\tau_n:=\inf \{t \geq 0:  \|\tilde{X}^1_t\| \geq n  \} \wedge\inf \{t \geq 0:  \|\tilde{X}^2_t\| \geq n  \}$.
For each $n \geq N+1$, let $\chi_n \in C_0^{\infty}(\mathbb{R}^d)$ be such that $0 \leq \chi_n \leq 1$ on $\mathbb{R}^d$ and $\chi_n \equiv 1$ on $\overline{B}_n$ and $\text{supp}(\chi_n) \subset B_{n+1}$. Set $\mathbf{G}_n:= \chi_n \mathbf{G}$ and for each $n \geq N+1$ and $1 \leq i,j \leq d$, define
$$
\sigma^n_{ij} := \chi_{n+1} \sigma_{ij} + \Lambda_{B_{n+1}}^{1/2} (1-\chi_n) \delta_{ij} \quad \text{ on }\, \mathbb{R}^d.
$$
Then, we have
\[
\lambda_{B_{n+1}} \| \xi \|^2 \leq   \| \sigma^n(x)^T \xi \|^2 \leq 4 \Lambda_{B_{n+1}} \| \xi \|^2.
\]
Meanwhile, observe that for each $n \geq N+1$ and $i \in \{1,2\}$
$$
\tilde{X}^i_t = x +  \int_0^t \tilde{\sigma}(\tilde{X}^i_s)  d \tilde{W}_s   +  \int_0^t \mathbf{G} (\tilde{X}^i_s) ds, \quad 0 \leq t< T \wedge \tau_n, \quad \text{$\tilde{\mathbb{P}}_x$-a.s.}
$$
By \cite[Theorem 1.1]{Z11},  for each $n \geq N+1$ it holds that
$$
\tilde{\mathbb{P}}_x  \left(  \tilde{X}^1_t= \tilde{X}_t^2 \quad \text{ for all $t \in [0, T \wedge \tau_n)$}	\right)=1.
$$
Letting $n \rightarrow \infty$ and using the left-continuity of $X^1$ and $X^2$ at $T$, the assertion follows.
\end{proof}

\noindent
The following lemma is a consequence of \cite[Proposition 2.2]{BRS21} and some calculations.
\begin{lem} \label{fundalem}
Let $x \in \mathbb{R}^d$, $T \in (0, \infty)$, and assume that {\bf (H-1)} and {\bf (H-2)} hold. Let  $(\mu_t)_{t \in [0, T]}$ be a family of probability measures on $\mathcal{B}(\mathbb{R}^d)$ such that $(\mu_t)_{t \in [0,T]}$ is a solution to the Cauchy problem for the Fokker–Planck equation associated with $L$ and the initial distribution $\delta_x$ as in Definition \ref{defnofoker}. Then,
$$
\int_0^T \int_{\mathbb{R}^d} \frac{\|A(y)\| + \langle \mathbf{G}(y), y  \rangle}{1+\|y\|^2} \mu_t(dy) dt <\infty.
$$
\end{lem}
\begin{proof}
Let $V(y)=\ln (1+\|y\|^2)$, $y \in \mathbb{R}^d$. Then, 
$$LV(y) = \frac{\text{trace}A(y) }{1+\|y\|^2} - \frac{2 \langle A(y)y, y\rangle}{(1+\|y\|^2)^2}+\frac{2\langle \mathbf{G}(y), y \rangle\rangle}{1+\|y\|^2} \quad \text{ for all $y \in \mathbb{R}^d$}.
$$
Then, for a.e. $y \in \mathbb{R}^d \setminus \overline{B}_N$, it follows that
\begin{align*}
LV(y) &\leq \frac{\text{trace}A(y) }{\|y\|^2} +\frac{2 \langle A(y)y, y\rangle}{\|y\|^4}+\frac{2\langle \mathbf{G}(y), y \rangle\rangle}{1+\|y\|^2} \\
&\leq \frac{K(d+2)}{N^2}+K(d+2)\cdot V(y) + \frac{2K}{1+N^2}+{2}K \cdot V(y) \\
&{\leq} K(d+{4})V(y) + \frac{K(d+4)}{N^2}.
\end{align*}
Now define $\tilde{K}:={K(d+4)}$ and
$$
W(y):=\frac{K(d+4)}{N^2}+ 1_{\overline{B}_{N_0}}(y) |LV(y)|, \quad y \in \mathbb{R}^d.
$$
Therefore, we finally obtain that
$$
LV(y) \leq W(y)  + \tilde{K} \cdot V(y) \quad \text{ for a.e. $y \in \mathbb{R}^d$}.
$$
Note that $\displaystyle \lim_{\|y\| \rightarrow \infty} V(y) = \infty$ and that
$$
\int_{0}^T \int_{\mathbb{R}^d} W(y) \mu_t(dy) dt <\infty
$$
from the fact that $\mathbf{G} \in L^1(B \times (0,T), \mu_t dt)$. Therefore,
it follows from \cite[Proposition 2.2]{BRS21} that
\begin{align}
&\quad \int_{\mathbb{R}^d} V d\mu_t \leq \left( V(x)+\int_0^T \int_{\mathbb{R}^d} W d\mu_s ds \right) e^{\tilde{K}T} \quad \text{ for all $t \in [0,T]$}, \label{inequ1} \\
& \int_0^T \int_{\mathbb{R}^d} |LV(y)| \mu_t(dy) dt \leq 2e^{\tilde{K}T} \left(	\int_0^T \int_{\mathbb{R}^d} W d\mu_s ds  +  V(x)	\right)<\infty. \label{inequ2}
\end{align}
Observe that from {\bf (H-2)} and \eqref{inequ1},
\begin{equation} \label{intermeineq}
\int_0^T \int_{\mathbb{R}^d} \frac{\|A(y)\|}{1+\|y\|^2} \mu_t(dy) dt \leq TK+{K}\int_0^T  \int_{\mathbb{R}^d} V(y) \mu_t(dy)dt< \infty.
\end{equation}
Moreover, it follows from \eqref{inequ2} and \eqref{intermeineq} that
\begin{align*}
&\int_0^T \int_{\mathbb{R}^d}  \frac{\left|\langle \mathbf{G}(y), y \rangle\rangle \right|}{1+\|y\|^2} \mu_t(dy)dt = \int_0^T \int_{\mathbb{R}^d}   \left| \frac{1}{2}LV(y) - \frac{\text{trace}A(y)}{2(1+\|y\|^2)}+\frac{2 \langle A(y)y, y\rangle}{(1+\|y\|^2)^2} \right| \mu_t(dy)dt  \\
& \leq \frac{1}{2}\int_0^T \int_{\mathbb{R}^d}   |LV(y)| \mu_t(dy)dt + \left(  \frac{d}{2}+2 \right)\int_0^T \int_{\mathbb{R}^d}   \frac{\|A(y)\|}{1+\|y\|^2} \mu_t(dy)dt <\infty. 
\end{align*}
Hence, the proof is complete.
\end{proof}

\begin{theo} \label{ouruniquene}
Let $x \in \mathbb{R}^d$, $T \in (0, \infty)$, and assume that {\bf (H-1)} and {\bf (H-2)} hold. 
Let $(\nu_t^1)_{t \in [0,T]}$ and $(\nu_t^2)_{t \in [0,T]}$ be solutions to the Fokker-Planck equations for $L$ with initial distribution $\delta_x$. Then, for each $t \in [0,T]$, $\nu^1_t = \nu^2_t$ on $\mathcal{B}(\mathbb{R}^d)$.
\end{theo}
\begin{proof}
Let $i \in \{1, 2\}$. By \cite[Theorem 1.1]{BRS21} and Lemma \ref{fundalem}, there exists a filtered probability space $(\bar{\Omega}_T, \mathcal{B}(\Omega_T), (\bar{\mathcal{F}}_{t})_{t \in [0,T]}, \bar{\mathbb{P}}^i_x)$ and a canonical process $(\bar{X}^{(i)}_t)_{t \in [0,T]}$ such that $\bar{\mathbb{P}}_x^i$ solves the martingale problem for $(L, C_0^{\infty}(\mathbb{R}^d))$ up to a finite time $T$ as in Definition \ref{martiprupto}, and satisfies
$$
\bar{\mathbb{P}}_x^i(\bar{X}^{(i)}_t \in E) = \nu^i_t(E) \; \text{ for all $E \in \mathcal{B}(\mathbb{R}^d)$ and $t \in [0, T]$}.
$$
Subsequently,  let $(\tilde{\Omega}^i, \tilde{\mathcal{F}}^i,  (\tilde{\mathcal{F}}^i)_{t \geq 0}, \tilde{\mathbb{P}}^i_x)$ be a filtered probability space and $(\tilde{X}^{(i)}_t)_{t \geq 0}$ be a canonical process as in Proposition \ref{prop12} such that (i)--(ii) in Proposition \ref{prop12} hold.
Then, we have
$$
\tilde{\mathbb{P}}^i_x (\tilde{X}^{(i)}_t \in E)=\bar{\mathbb{P}}_x^i(\bar{X}^{(i)}_t \in E) = \nu^i_t(E) \; \text{ for all $E \in \mathcal{B}(\mathbb{R}^d)$ and $t \in [0, T]$}.
$$
Moreover, by Theorem \ref{martiweak}, 
 there exists an extension $(\hat{\Omega}, \hat{\mathcal{F}}, \hat{\mathbb{P}}_x, 
 (\hat{\mathcal{F}_t})_{t \geq 0}, (\hat{X}_t)_{t \geq 0})$ of \\
 $(\tilde{\Omega}, \tilde{\mathcal{F}}, \tilde{\mathbb{P}}_x, (\tilde{\mathcal{F}_t})_{t \geq 0}, (\tilde{X}_t)_{t \geq 0})$, and there exists an $(\hat{\mathcal{F}}^i_t)_{t \geq 0}$-standard Brownian motion $(\hat{W}^{(i)}_t)_{t \geq 0}$
such that 
$$
\hat{X}^{(i)}_t = x+ \int_0^t \sigma(\hat{X}^{(i)}_s) d \hat{W}^{(i)}_s + \int_0^t \mathbf{G}(\hat{X}^{(i)}_s) ds, \quad 0 \leq t \leq T, \qquad \hat{\mathbb{P}}_x\text{-a.s.}
$$
and that 
$$
\hat{\mathbb{P}}^i_x (\hat{X}^{(i)}_t \in E)=\tilde{\mathbb{P}}^i_x (\tilde{X}^{(i)}_t \in E)=\bar{\mathbb{P}}_x^i(\bar{X}^{(i)}_t \in E) = \nu^i_t(E) \; \text{ for all $E \in \mathcal{B}(\mathbb{R}^d)$ and $t \in [0, T]$}.
$$
Finally, by the consequence of Theorem \ref{pathwiseuni} and the Yamada-Watanabe's theorem in \cite[Chapter 5, Proposition 3.20]{KS91} where $C([0, \infty), \mathbb{R}^d)$ is replaced by $C([0, T], \mathbb{R}^d)$, we discover that
$$
\nu^1_t(E) = \hat{\mathbb{P}}_x^1(\tilde{X}^{(1)}_t \in E)=\hat{\mathbb{P}}_x^2(\tilde{X}^{(2)}_t \in E) = \nu_t^2(E) \quad \text{ for all $t \in [0,T]$ and $E \in \mathcal{B}(\mathbb{R}^d)$}.
$$
This completes the proof.
\end{proof}

\section{Conclusions and discussion} \label{condis}
\noindent
The solutions to the Fokker-Planck equations can be expressed as the one-dimensional marginal distributions of the solutions to stochastic differential equations. This allows probabilistic objects to be interpreted as deterministic objects in the form of solutions to partial differential equations. In this paper, using semigroup theory, Cholesky decomposition and stochastic analysis, we show the existence and uniqueness of solutions to the Cauchy problem for the Fokker-Planck equations associated with a partial differential operator $L$ with highly irregular coefficients. Furthermore, we provided results on the long-time behavior of the solutions, demonstrating their ergodicity.
Compared to prior works such as \cite{Fi08} and \cite{Tr16}, this work distinguishes itself by establishing the well-posedness of the Cauchy problem for the Fokker–Planck equation associated with operators whose drift coefficients are not assumed to be weakly differentiable and may even exhibit locally unbounded singularities that blow up to infinity. 
\\
As highlighted in Remark \ref{growthconrema}, the growth condition expressed as inequalities in {\bf (H-2)} is observed to be slightly stringent for guaranteeing the existence of solutions to the Cauchy problem for the Fokker-Planck equations. Thus, further investigation is needed to determine whether uniqueness can be established under the weaker condition as \eqref{lapunoconser}. This requires additional research into whether the superposition principle developed in \cite{BRS21} can be generalized under the condition \eqref{lapunoconser}  (cf. \cite[Question 8]{BRS23}).\\
Recently, the application of Fokker-Planck equations in Artificial Intelligence, particularly in image generation algorithms, has attracted significant attention (cf. \cite{SS21}). The philosophy of treating images as samples drawn from specific probability distributions is fundamentally rooted in the ideas of Markov Chain Monte Carlo (MCMC) algorithms. This raises several questions about how the family of probability measures expressed as solutions to the Fokker-Planck equations evolves over time, how they converge to a target distribution as time grows large, the rate of convergence, and how to quantify the distance between probability measures to assess this convergence. Notably, recent studies have revealed that when non-symmetricity is imposed on the partial differential operator $L$, the convergence to the stationary probability measure of the Fokker-Planck equations accelerates (cf. \cite{HHS05, MCF15}).  To identify specific conditions that facilitate accelerated convergence, a systematic investigation into the well-posedness of Fokker-Planck equations under more generalized coefficient conditions would be critically important. As a direction for future research, we note that the limiting measure of the solution to the Cauchy problem coincides with the solution to the corresponding stationary Fokker–Planck equation (see Theorem \ref{mainresu}(iii)). Therefore, a systematic study of not only the existence but also the regularity and uniqueness of solutions to the stationary problem is required (see \cite{LT22, L25} for related results). Furthermore, there is growing interest in physics and engineering in Fokker–Planck equations corresponding to Stratonovich type SDEs (see \cite{AGT14}) and in nonlinear Fokker–Planck equations (see \cite{Lima22, BRS19}), both of which are important directions that deserve further investigation.

\text{}\\ \\ \\
\noindent
{\bf Acknowledgment.}\;
The author would like to thank the anonymous reviewers for their valuable comments and suggestions.

\text{}\\

\centerline{}
\centerline{}
\centerline{}
\centerline{}
Haesung Lee\\
Department of Mathematics and Big Data Science,  \\
Kumoh National Institute of Technology, \\
Gumi, Gyeongsangbuk-do 39177, Republic of Korea, \\
E-mail: fthslt@kumoh.ac.kr, \; fthslt14@gmail.com

\begin{thebibliography}{XXX}

\bibitem{A82}
B.D.O. Anderson, {\it Reverse-time diffusion equation models}, Stochastic Process. Appl., 12 (1982), 313–326.

\bibitem{AK08}
G. Allaire, S.M. Kaber, {\it Numerical linear algebra}, Texts Appl. Math., 55, Springer, New York, 2008.

\bibitem{BKR01}
V.I. Bogachev, N.V. Krylov, M. R\"{o}ckner, {\it On regularity of transition probabilities and invariant measures of singular diffusions under minimal conditions}, Comm. Partial Differential Equations, 26 (2001), 2037–2080.

\bibitem{BPRS07}
V.I. Bogachev, G. Da Prato, M. R\"{o}ckner, W. Stannat, {\it Uniqueness of solutions to weak parabolic equations for measures}, Bull. Lond. Math. Soc., 39 (2007), 631–640.

\bibitem{BKRS15}
V.I. Bogachev, N.V. Krylov, M. R\"{o}ckner, S. Shaposhnikov, {\it Fokker-Planck-Kolmogorov equations}, Math. Surveys Monogr., 207, American Mathematical Society, Providence, RI, 2015.

\bibitem{BKS21}
V.I. Bogachev, T.I. Krasovitskii, S.V. Shaposhnikov, {\it On uniqueness of probability solutions of the Fokker-Planck-Kolmogorov equation}, Mat. Sb., 212 (2021), 3–42.

\bibitem{BRS19}
V.I. Bogachev, M. R\"{o}ckner, S.V. Shaposhnikov, {\it Convergence in variation of solutions of nonlinear Fokker–Planck–Kolmogorov equations to stationary measures}, J. Funct. Anal., 276 (2019), no. 12, 3681–3713.

\bibitem{BRS21}
V.I. Bogachev, M. R\"{o}ckner, S.V. Shaposhnikov, {\it On the Ambrosio-Figalli-Trevisan superposition principle for probability solutions to Fokker-Planck-Kolmogorov equations}, J. Dynam. Differential Equations, 33 (2021), 715–739.

\bibitem{BRS23}
V.I. Bogachev, M. R\"{o}ckner, S.V. Shaposhnikov, {\it Kolmogorov problems on equations for stationary and transition probabilities of diffusion processes}, Theory Probab. Appl., 68 (2023), no. 3, 342–369.

\bibitem{CH97}
Z. Chen, Z. Huan, {\it On the continuity of the $m$th root of a continuous nonnegative definite matrix-valued function}, J. Math. Anal. Appl., 209 (1997), 60–66.

\bibitem{F06}
A. Friedman, {\it Stochastic differential equations and applications}, Dover Publications, Inc., Mineola, NY, 2006.

\bibitem{Fi08}
A. Figalli, {\it Existence and uniqueness of martingale solutions for SDEs with rough or degenerate coefficients}, J. Funct. Anal., 254 (2008), 109–153.

\bibitem{AGT14}
Z. González Arenas, D.G. Barci, C. Tsallis, {\it Nonlinear inhomogeneous Fokker-Planck equation within a generalized Stratonovich prescription}, Phys. Rev. E, 90 (2014), 032118.

\bibitem{HHS05}
C.R. Hwang, S.Y. Hwang-Ma, S.J. Sheu, {\it Accelerating diffusions}, Ann. Appl. Probab., 15 (2005), 1433–1444.

\bibitem{IW89}
N. Ikeda, S. Watanabe, {\it Stochastic Differential Equations and Diffusion Processes}, 2nd edn., North-Holland Mathematical Library, vol. 24, North-Holland Publishing Co., Amsterdam; Kodansha, Ltd., Tokyo, 1989.

\bibitem{KS91}
I. Karatzas, S. Shreve, {\it Brownian Motion and Stochastic Calculus}, 2nd edn., Graduate Texts in Mathematics, vol. 113, Springer, New York, 1991.

\bibitem{LT21}
H. Lee, G. Trunau, {\it Existence, uniqueness and ergodic properties for time-homogeneous It\^{o}-SDEs with locally integrable drifts and Sobolev diffusion coefficients}, Tohoku Math. J., 73 (2021), 159–198.

\bibitem{LT22}
H. Lee, G. Trutnau, {\it Existence and uniqueness of (infinitesimally) invariant measures for second order partial differential operators on Euclidean space}, J. Math. Anal. Appl., 507 (2022), 125778.

\bibitem{LST22}
H. Lee, W. Stannat, G. Trunau, {\it Analytic theory of It\^{o}-stochastic differential equations with non-smooth coefficients}, SpringerBriefs Probab. Math. Stat., Springer, Singapore, 2022.

\bibitem{L25}
H. Lee, {\it Local elliptic regularity for solutions to stationary Fokker-Planck equations via Dirichlet forms and resolvents}, Boundary Value Problems, 2025 (2025), Article number: 68.

\bibitem{Lima22}
L.S. Lima, {\it Interplay between nonlinear Fokker-Planck equation and stochastic differential equation}, Probabilistic Engineering Mechanics, 68 (2022), 103201.

\bibitem{MCF15}
Y.A. Ma, T. Chen, E. Fox, {\it A complete recipe for stochastic gradient MCMC}, Advances in neural information processing systems 28 (2015).

\bibitem{P14}
G.A. Pavliotis, {\it Stochastic processes and applications}, Texts Appl. Math., 60, Springer, New York, 2014.

\bibitem{RZ10}
M. R\"{o}ckner, X. Zhang, {\it Weak uniqueness of Fokker-Planck equations with degenerate and bounded coefficients}, C. R. Math. Acad. Sci. Paris, 348 (2010), 435–438.

\bibitem{Sh12}
S.V. Shaposhnikov, {\it On the uniqueness of the probabilistic solution of the Cauchy problem for the Fokker-Planck-Kolmogorov equation}, Theory Probab. Appl., 56 (2012), 96–115.

\bibitem{SS21}
Y. Song, J. Sohl-Dickstein, D.P. Kingma, A. Kumar, S. Ermon, B. Poole, {\it Score-based generative modeling through stochastic differential equations}, in Proc. Int. Conf. Learn. Representations, (2021).

\bibitem{SV06}
D.W. Stroock, S.R.S. Varadhan, {\it Multidimensional Diffusion Processes}, Springer, Berlin, 2006.

\bibitem{Tr16}
D. Trevisan, {\it Well-posedness of multidimensional diffusion processes with weakly differentiable coefficients}, Electron. J. Probab., 21 (2016), 41 pp.

\bibitem{YW71}
T. Yamada, S. Watanabe, {\it On the uniqueness of solutions of stochastic differential equations}, J. Math. Kyoto Univ., 11 (1971), 155–167.

\bibitem{Z11}
X. Zhang, {\it Stochastic homeomorphism flows of SDEs with singular drifts and Sobolev diffusion coefficients}, Electron. J. Probab., 16 (2011), 1096–1116.
\end{thebibliography}
\end{document}